\documentclass[11pt]{article}
\usepackage{mathrsfs}
\usepackage{latexsym}
\usepackage{amsmath,amssymb}
\usepackage[pdftex]{hyperref}
\usepackage{amsthm}
\usepackage{amsfonts,cite}
\usepackage[usenames]{color}
\usepackage{amssymb}
\usepackage{graphicx}
\usepackage{amsmath}
\usepackage{amsfonts}
\usepackage{amsthm}
\usepackage{mathrsfs}
\usepackage{dsfont}
\newcommand\be{\begin{equation}}
\newcommand\ee{\end{equation}}
\newcommand\ber{\begin{eqnarray}}
\newcommand\eer{\end{eqnarray}}
\newcommand\berr{\begin{eqnarray*}}
\newcommand\eerr{\end{eqnarray*}}
\newcommand\bea{\begin{eqnarray}}
\newcommand\eea{\end{eqnarray}}
\newcommand\ba{\begin{array}}
\newcommand\ea{\end{array}}

\newcommand{\nn}{\nonumber}

\newcommand{\bi}{\begin{itemize} }
  \newcommand{\ei}{\end{itemize} }

\newtheoremstyle{mythm}{1.5ex plus 1ex minus .2ex}{1.5ex plus 1ex
minus .2ex}{\kai}{\parindent}{\song\bfseries}{}{1em}{}
\numberwithin{equation}{section}\numberwithin{figure}{section}

\newtheorem{theorem}{Theorem}[section]
\newtheorem{lemma}{Lemma}[section]

\newtheorem{proposition}{Proposition}[section]

\allowdisplaybreaks[4]
\begin{document}

\title{ Vortex Solutions for A Mixed Boundary-Value Problem in the Abelian-Higgs Model  with A Neutral Scalar Field}
\author{Guange Su\\Department of Mathematics and Physics \\Luoyang Institute of Science and Technology\\
Luoyang, Henan 471023, PR China\\Xiaosen Han\textsuperscript{*}\\School of Mathematics and Statistics\\Henan University\\
Kaifeng, Henan 475004, PR China}
\date{}
\maketitle

\begin{abstract}
Vortices represent a class of topological solitons arising in gauge theories coupled with complex scalar fields, holding significant importance across various domains of modern physics. In this paper we establish  the existence of vortex solutions for a mixed boundary-value problem derived from the Abelian-Higgs model incorporating a neutral scalar field, a system recently investigated by Eto, Peterson et al. [7]. By synergistically combining the shooting method with the Schauder fixed-point theorem, we derive sharp analytical criteria that delineate the Abelian vortex phase from the non-Abelian one. We also rigorously establish the monotonicity, uniform boundedness, and precise asymptotic behavior of the vortex profile functions. Our results provide rigorous confirmation of numerical observations regarding the phase boundary between these distinct vortex types.

\end{abstract}

\medskip
\begin{enumerate}

\item[]

\end{enumerate}

\section{Introduction}
Vortex has important applications in many areas of modern physics including superfluids, superconductors~\cite{pm,sb}~and holographic superconductors~\cite{wyw,wsm}~as well as cosmic strings~\cite{va,hbm}~in high-energy physics. In many physical theories, vortices often play a crucial role in understanding key phenomena, greatly increasing their importance. Mathematically, it is a topological soliton under the complex scalar field coupled gauge field model~\cite{yy}.  Research on vortex solutions in classical and quantum field theory has yielded a rich body of work, encompassing examples such as the simple Kirchhoff vortex in fluid dynamics, magnetic vortices in Ginzburg-Landau superconductivity theory, gravitational vortices arising from the coupling of Einstein's theory of gravity with the Abelian-Higgs field theory, and non-Abelian vortices featured in the Glashow-Salam-Weinberg unified electroweak theory. The study of vortices not only aids in explaining physical phenomena but also frequently stimulates the development of novel mathematical ideas and techniques.

Magnetic vortex configurations were investigated by Abrikosov in the context of Ginzburg-Landau theory of superconductivity \cite{aa}. Based on this result, Nambu and 't Hooft proposed a theory of double superconductivity to explain the confinement of quarks \cite{ny,th,ms}. But it is well known that quark confinement is a non-Abelian gauge field problem. Therefore, using this model to study the quark color limit problem is definitely not perfect. Seiberg and Witten found that magnetic monopoles do condense in a softly broken $\mathcal{N}=2$ Super Yang-Mills theory \cite{sn}. Their condensation does in fact yield electric flux tubes, which is one of the main desired ingredients. However, the mechanism behind this remains purely Abelian and cannot describe any actual physical properties of Nature's quarks.  It therefore becomes important to study non-Abelian vortex dynamics. Non-Abelian vortices are standard solitonic vortex solutions with additional orientational moduli localized at their cores \cite{ha,ar}. These moduli are the Goldstone bosons resulting from the breaking of a non-Abelian global symmetry in the vortex core \cite{ckm,dht,ht,sy,smtg}. The important application background and the mathematical difficulties caused by nonlinearity have aroused people's strong research interest. In recent years, research on non-Abelian vortices has also yielded rich results \cite{chls,hl,ly1,ly2,ly3,gjk,tg,ge,xc}.

Inspired by Witten's superconducting string model\cite{we}, Shifman has proposed a particulary simple model which also contains non-Abelian vortices, free of any supersymmetry \cite{sm}. Eto, Peterson and others  subsequently modified Shifman's original model and derived a mixed boundary value problem on the semi-infinite interval $r\in[0,\infty)$ \cite{em}. In this paper, we solve this problem by rigorous mathematical analysis. Our difficulty lies in the existence of non-Abelian vortex which we use to describe those vortices possessing a non-zero condensate in the core. To overcome this difficulty, we combine shooting method and fixed-point theorem argument, which was first proposed by \cite{mbj,mbj1}. We analytically obtain the parameter region such that $\chi$-condensation at the vortex core condensation is non-zero. Namely, we clarify when the vortex becomes of the non-Abelian type. The vortex solutions essentially depend on the parameter $\alpha$. If $\alpha\geq\frac{1}{2}$, there is only Abelian vortex. However, if $0<\alpha<\frac{1}{2}$, non-Abelian vortex may appear. Furthermore, we show that the field profiles of the vortex solution to the equations of motion are monotonic functions.

The rest of our paper is organized as follows. In Section 2, we introduce the Abelian
Higgs model with neutral scalar field, and then state our main existence theorem for vortex solutions. In Sections 3, 4, and 5, we prove the existence and uniqueness of the respective profile functions, and we prove the existence of the vortex solution for each second order nonlinear ordinary differential equation via a shooting argument. Moreover, we obtain some qualitative properties of these solutions. Finally, in Section 6, we complete the proof of main existence theorem by using the Schauder fixed point theorem.

\section{Model and main results}
The system of interest is an Abelian-Higgs model with neutral scalar fields. The matter contents are the $U(1)$ gauge field $A_{\mu}$, the charged Higgs field $\phi$, and the neutral real scalar fields $\chi_{a}(a=1,\cdots,N)$. Following \cite{em}, we consider one of the simplest models which admits a nontrivial vortex with internal orientational moduli. The Lagrangian is
\begin{equation}\label{0.1}
\mathscr{L}=-\frac{1}{4}F_{\mu\nu}F^{\mu\nu}+D_{\mu}\phi(D^{\mu}\phi)^{*}+\partial_{\mu}\chi_{a}\partial^{\mu}\chi_{a}-V.
\end{equation}
Here field strength tensor $F_{\mu\nu}=\partial_{\mu}A_{\nu}-\partial_{\nu}A_{\mu}$, covariant derivative $D_{\mu}\phi=(\partial_{\mu}+ieA_{\mu})\phi$, $e$ is a $U(1)$ charge, and the potential is defined as
\begin{equation}\label{0.2}
V=\Omega^{2}\chi^{2}_{a}+\lambda(|\phi|^{2}+\chi^{2}_{a}-v^{2})^{2},
\end{equation}
where $\Omega^{2}>0$, $v^{2}>0$, $\lambda>0$ are coupling constants. Then the potential minimum reads
\ber
|\phi|=v, ~\vec{\chi}=0.
\eer
This is the Higgs branch in the sense that the $U(1)$ gauge symmetry is broken. On the other hand, the global $O(N)$ symmetry is unbroken.  Instead, $(x,y)=(v,0)$ is always a local minimum where all bosonic excitations are massive and
\ber
m^{2}_{\phi}=4v^{2}\lambda,~~m^{2}_{\gamma}=2e^{2}v^{2},~~m^{2}_{\chi}=\Omega^{2}.
\eer

When the gauge coupling is turned on, the gauge fields become dynamical fields. The full equations of motion resulting from Lagrangian (\ref{0.1}) with potential (\ref{0.2}) are
\begin{align}
\partial^{\mu}\partial_{\mu}A^{\nu}-\partial^{\nu}\partial_{\mu}A^{\mu}+ie(\phi D^{\nu}\phi^{\star}-\phi^{\star}D^{\nu}\phi)&=0,\label{0.4}\\
D_{\mu}D^{\mu}\phi+\frac{\partial V}{\partial\phi^{\star}}&=0,\label{0.5}\\
\partial_{\mu}\partial^{\mu}\chi^{i}+\frac{1}{2}\frac{\partial V}{\partial \chi^{i}}&=0.\label{0.6}
\end{align}

We seek solutions of this system within the following ansatz
\begin{equation}\label{0.7}
\phi=vf(r)e^{i\theta},~~\chi=vg(r)\boldsymbol{n},~~A^{i}=\frac{1}{e}\epsilon^{ij}\frac{x^{j}}{r^{2}}(1-a(r)),
\end{equation}
where $r$ and $\theta$ are the two-dimensional polar coordinates, $\boldsymbol{n}$ is an arbitrary constant unit vector, and $f(r)$, $g(r)$, $a(r)$ are the real profile functions. Substituting (\ref{0.7}) into (\ref{0.4})-(\ref{0.6}) yields a coupled system of
ordinary differential equations
\begin{align}
a''-\frac{1}{r}a'-\beta^{2}f^{2}a&=0,\label{1.1}\\
g''+\frac{1}{r}g'-\alpha^{2}g-\frac{1}{2}(f^{2}+g^{2}-1)g&=0,\label{1.2}\\
f''+\frac{1}{r}f'-\frac{a^{2}}{r^{2}}f-\frac{1}{2}(f^{2}+g^{2}-1)f&=0,\label{1.3}
\end{align}
subject to the following boundary conditions
\begin{align}
(a,g',f)\rightarrow&(1,0,0)~~~~~\text{for }r\rightarrow0,\label{1.4}\\
(a,g,f)\rightarrow&(0,0,1)~~~~~\text{for }r\rightarrow\infty.\label{1.5}
\end{align}
Here $\alpha=\frac{m_{\chi}}{m_{\phi}}>0$, $\beta=\frac{m_{\gamma}}{m_{\phi}}>0$. We call a vortex non-Abelian only if the $\chi$-condensation at the vortex center satisfies $g(0)>0$, otherwise we regard it as Abelian.
The main results of this paper are stated as follows.
\begin{theorem}\label{t1}
For any parameters $\alpha>0$ and $\beta>0$, the equations of motion of the Abelian-Higgs model with neutral scalar field defined by the Lagrangian density (\ref{0.1}) have a radially symmetric solution described by the ansatz (\ref{0.7}) so that $(a, g, f)$ satisfies the boundary conditions (\ref{1.4}) and (\ref{1.5}), $a, f$ are strictly monotone functions of $r$, $0<a(r)<1$, $0<f(r)<1$ for all $r>0$. In particularly,
\item(i) if $\alpha\geq\frac{1}{2}$, then there is only Abelian vortex, namely, $g(r)\equiv0$ for all $r>0$.
\item(ii) if $\alpha<\frac{1}{2}$, then there exists non-Abelian vortex. More precisely, $g(r)$ is strictly decreasing and for all $r>0$,
\ber
0<g(r)<\sqrt{1-2\alpha^{2}},~0<f^{2}(r)+g^{2}(r)<1.
\eer
Furthermore, there hold the asymptotic estimates
\begin{align}
a(r)&=1+O(r^{2}),\\
g(r)&=g(0)+O(r^{2}),\\
f(r)&=O(r),
\end{align}
for $r\rightarrow0$, where $0<g(0)<\sqrt{1-2\alpha^{2}}$, and
\begin{align}
a(r)&=O(r^{\frac{1}{2}}\exp{(-\beta(1-\varepsilon)r)}),\\
g(r)&=O(\exp{(-\alpha(1-\varepsilon)r)}),\\
f(r)&=1+O(\exp{(-\gamma(1-\varepsilon)r)}),
\end{align}
for $r\rightarrow\infty$, where $0<\varepsilon<1$ is arbitrary and the exponents $\gamma=\min\{1,2\alpha,2\beta\}$.
\end{theorem}
It is easily seen that the phase boundary between the Abelian and non-Abelian phases stands to the line $\alpha=\sqrt{\frac{1}{2}}$, which rigorously confirms the numerical results obtained in \cite{em}. The steps to prove Theorem \ref{t1} will be described as follows. Firstly, fixed a function $f(r)$ satisfying the boundary condition $f(0)=0$, $f(\infty)=1$ and some additional properties, we obtain a unique decreasing solution $a(r)$ with $a(0)=1$, $a(\infty)=0$ via applying shooting method to (\ref{1.1}). For the fixed $f(r)$, we use shooting method applied to (\ref{1.2}) to get a unique $g(r)$ with $g'(0)=0$, $g(\infty)=0$. Next, substituting the obtained functions $a(r)$ and $g(r)$ into (\ref{1.3}), we get a unique increasing function $\tilde{f}(r)$ satisfying (\ref{1.3}) and $\tilde{f}(0)=0$, $\tilde{f}(\infty)=1$. Technically, the process of the three-step shooting allow us to define a mapping $T$, which sends $f$ to $\tilde{f}$ : $T(f)=\tilde{f}$. We show that $T$ maps some convex bounded set into itself and is compact. Therefore, Schauder fixed point theorem guarantees that $T$ has a fixed point $f(r)$. In this way, we get a solution to the boundary value problem (\ref{1.1})-(\ref{1.5}). The above steps will be presented in the following sections in details.

\section{Existence and uniqueness of $a(r)$}
In this section, we are to show the existence and uniqueness of $a(r)$ for a suitably fixed function $f(r)$ via a single parameter shooting method. The result of this section can be described as the following proposition.
\begin{proposition}\label{l1.1}
Given function $f(r)\in C[0,\infty)$ such that $f$ is increasing,
$r^{-1}f\leq M_{1}$ for $r\leq1$ $(M_{1}$ is a positive constant), $f(0)=0$ and $f(\infty)=1$, then we can find a unique function $a(r)$ satisfying (\ref{1.1}) with boundary condition $a(0)=1$, $a(\infty)=0$. Moreover, $a(r)$ is strictly decreasing, $0<a(r)<1$ for all $r>0$, $|a-1|\leq M_{2}r^{2}$ for any $r\leq1$, where $M_{2}$ is a positive constant.
\end{proposition}

To proceed, we define $\tilde{a}(r)=1-a(r)$ and consider the differential equation
\begin{equation}\label{2.1}
\tilde{a}''-\frac{1}{r}\tilde{a}'+\beta^{2}f^{2}(1-\tilde{a})=0,
\end{equation}
which can be formally transformed into the integral equation
\begin{equation}\label{2.2}
\tilde{a}(r)=Br^{2}+\frac{\beta^{2}}{2}\int_{0}^{r}s\left(\frac{r^{2}}{s^{2}}-1\right)f^{2}(s)(\tilde{a}(s)-1)\mathrm{d}s.
\end{equation}
Here $B>0$ is an arbitrary constant. We first demonstrate that for any $B>0$, there exists a local solution to the equation (\ref{2.1}) satisfying the initial condition $\tilde{a}(0)=0$.

\begin{lemma}\label{l1}
Let $f(r)$ be given in Proposition \ref{l1.1} and $B>0$, the equation (\ref{2.1}) admits a unique solution near $r=0$ satisfying $\tilde{a}(0)=0$, more precisely,
\ber\label{2.3}
\tilde{a}(r)=Br^{2}+O(r^{4})~(r\rightarrow0).
\eer
\end{lemma}
\begin{proof}
We employ Picard iteration to solve the integral equation (\ref{2.2}). To begin with, set
\begin{align}
\tilde{a}_{0}(r)&=Br^{2},\\
\tilde{a}_{n+1}(r)&=Br^{2}+\frac{\beta^{2}}{2}\int_{0}^{r}s\left(\frac{r^{2}}{s^{2}}-1\right)f^{2}(s)(\tilde{a}_{n}(s)-1)\mathrm{d}s,~~~n=0,1,2,...
\end{align}
Since $f(r)\leq M_{1}r$ for any $r\leq1$, we have
\begin{align}\label{2.4}
|\tilde{a}_{1}-\tilde{a}_{0}|&\leq\frac{\beta^{2}}{2}\int_{0}^{r}\frac{r^{2}}{s}M_{1}^{2}s^{2}(Bs^{2}+1)\mathrm{d}s\notag\\
&\leq\frac{\beta^{2}}{4}M_{1}^{2}r^{4}+\frac{\beta^{2}}{8}M_{1}^{2}Br^{6},\\
&\leq\frac{3\beta^{2}}{8}M_{1}^{2}r^{4},~~~r\leq\min{\left\{1,\sqrt{\frac{1}{B}}\right\}}.\notag
\end{align}
Then
\begin{equation}
|\tilde{a}_{1}|\leq Br^{2}+\frac{3\beta^{2}}{8}M_{1}^{2}r^{4}\leq M_{*}r^{2},~~~r\leq\min{\left\{1,\sqrt{\frac{1}{B}}\right\}},\notag
\end{equation}
where $M_{*}=B+\frac{3\beta^{2}}{8}M_{1}^{2}$. Similarly, for any $r\leq\min{\left\{1,\sqrt{\frac{1}{B}},\sqrt{\frac{1}{M_{*}}}\right\}}$, we obtain
\begin{align}
|\tilde{a}_{2}-\tilde{a}_{0}|&\leq\frac{\beta^{2}}{2}\int_{0}^{r}\frac{r^{2}}{s}M_{1}^{2}s^{2}(M_{*}s^{2}+1)\mathrm{d}s\notag\\
&\leq\frac{\beta^{2}}{4}M_{1}^{2}r^{4}+\frac{\beta^{2}}{8}M_{1}^{2}M_{*}r^{6}\notag\\
&\leq\frac{3\beta^{2}}{8}M_{1}^{2}r^{4},\notag
\end{align}
so that $|\tilde{a}_{2}|\leq M_{*}r^{2}$. As a consequence, setting
\begin{equation}\label{00}
\delta=\min{\left\{1,\sqrt{\frac{1}{B}},\sqrt{\frac{1}{M_{*}}}\right\}}, 
\end{equation}
we see that for any $r\leq\delta$,
\begin{align}
|\tilde{a}_{n}-\tilde{a}_{0}|&\leq\frac{3\beta^{2}}{8}M_{1}^{2}r^{4},~~~n=1,2,3,...\label{2.5}\\
|\tilde{a}_{n}|&\leq M_{*}r^{2},~~~n=1,2,3,...\label{2.6}
\end{align}

On the other hand, a direct calculation shows that for $r\leq\delta$,
\begin{align}
|\tilde{a}_{n+1}-\tilde{a}_{n}|&\leq\frac{\beta^{2}}{2}\int_{0}^{r}s\left(\frac{r^{2}}{s^{2}}-1\right)f^{2}(s)|\tilde{a}_{n}(s)-\tilde{a}_{n-1}(s)|\mathrm{d}s\notag\\
&\leq\frac{\beta^{2}}{2}\int_{0}^{r}\frac{r^{2}}{s}|\tilde{a}_{n}(s)-\tilde{a}_{n-1}(s)|\mathrm{d}s\notag\\
&\leq \frac{3\beta^{2}}{8}M_{1}^{2}r^2\left(\frac{\beta^{2}}{2}\right)^{n}\frac{2r^{2(n+1)}}{2(n+1)!!}\notag\\
&\leq\left(\frac{\beta^{2}}{2}\right)^{n}\frac{2r^{2(n+1)}}{2(n+1)!!} ~~~n=0,1,2,3,...\notag
\end{align}
Hence,
\ber\label{2.7}
\sum_{n=0}^{\infty}(\tilde{a}_{n+1}-\tilde{a}_{n})\leq\sum_{n=0}^{\infty}\left(\frac{\beta^{2}}{2}\right)^{n}\frac{2r^{2(n+1)}}{2(n+1)!!},~~r\leq\delta.
\eer
In view of (\ref{2.4}) and (\ref{2.7}), it is clear that the series $\tilde{a}_{0}(r)+\sum\limits_{n=1}^{\infty}(\tilde{a}_{n}(r)-\tilde{a}_{n-1}(r))$ converges uniformly on $[0,\delta]$, which implies that the sequence $\{\tilde{a}_{n}\}$ converges. We may assume $\lim\limits_{n\rightarrow\infty}\tilde{a}_{n}(r)=\tilde{a}(r)$, so that $\tilde{a}(r)$ is a solution to problem (\ref{2.1}) over $[0,\delta]$. Moreover, $\tilde{a}(r)$ can be extended uniquely to a maximal interval $[0, R_{B})$ such that either $R_{B}=\infty$ or $|\tilde{a}(r)|\rightarrow\infty$ as $r\rightarrow R_{B}$. Taking $n\rightarrow\infty$ in (\ref{2.5}), we obtain (\ref{2.3}).
\end{proof}

To invoke a shotting argument, we denote the solution as $\tilde{a}(B,r)$ and set
\ber
\begin{aligned}
\mathcal{B}^{-}=&\big\{B>0~|~\text{there exists a}~r_{*}\in(0,R_{B})~\text{such that}
~\tilde{a}'(B,r_{*})<0~\text{and for any}~r\in(0,r_{*}],0<\tilde{a}(B,r)<1\big\},\notag\\
\mathcal{B}^{0}=&\big\{B>0~|~\tilde{a}'(B,r)\geq0\mbox{ and } 0<\tilde{a}(B,r)<1~\text{for all}~r>0\big\},\notag\\
\mathcal{B}^{+}=&\big\{B>0~|~\mbox{there exists a}~r^{*}\in(0,R_{B})~\text{such that}~\tilde{a}(B,r^{*})=1~\text{and for any}~r\in(0,r^{*}], \tilde{a}'(B,r)\geq0\big\}.
\end{aligned}
\eer

\begin{lemma}\label{l.0}
$\mathcal{B}^{-}\cap\mathcal{B}^{0}=\mathcal{B}^{0}\cap\mathcal{B}^{+}=\mathcal{B}^{-}\cap\mathcal{B}^{+}=\emptyset$ and $\mathcal{B}^{-}\cup\mathcal{B}^{0}\cup\mathcal{B}^{+}=(0,\infty)$.
\end{lemma}
\begin{proof}
Obviously, $\mathcal{B}^{-}$, $\mathcal{B}^{0}$ and $\mathcal{B}^{+}$ are nonoverlapping. If there are some value of $B>0$ and $B\notin\mathcal{B}^{-}$, then only three cases may happen. Case(i), $\tilde{a}'(B,r)\geq0$ and $0<\tilde{a}(B,r)<1$ for all $r>0$, namely, $B\in\mathcal{B}^{0}$. Case(ii), $\tilde{a}'(B,r)\geq0$ for all $r>0$ and $\tilde{a}(B,r_{0})\geq1$ for some $r_0$. We claim that $\tilde{a}'(B,r_{0})=0$ and $\tilde{a}(B,r_{0})=1$ cannot happen simultaneously. Otherwise, applying the uniqueness theorem for the initial value problem of ordinary differential equations, we have $\tilde{a}(B,r)\equiv1$ on $(0,\infty)$, which contradicts the fact $\tilde{a}(B,0)=0$. Thus $a\in\mathcal{B}^{+}$.  Case(iii), there exists $r_1$, first zero of $\tilde{a}'(B,r)$, such that $\tilde{a}(B,r_{1})>1$. Then $a\in\mathcal{B}^{+}$. Therefore $\mathcal{B}^{-}\cup\mathcal{B}^{0}\cup\mathcal{B}^{+}=(0,\infty)$ and the lemma follows.
\end{proof}

\begin{lemma}\label{l3.5}
Let $f(r)$ be given in Proposition \ref{l1.1}. If $B_{0}\in\mathcal{B}^{0}$, then $\tilde{a}(B_{0},r)$ is a solution to the differential equation (\ref{2.1}) satisfying boundary condition $\tilde{a}(B_{0},0)=0$, $\tilde{a}(B_{0},\infty)=1$.

\end{lemma}
\begin{proof}
We need only to show that $\tilde{a}(B_{0},r)\rightarrow1$ as $r\rightarrow\infty$. Since $\tilde{a}(B_{0},r)$ is bounded and strictly increasing, the limit $\lim\limits_{r\rightarrow\infty}\tilde{a}(B_{0},r)$ must exist.

Indeed, the equation (\ref{2.1}) can be rewrited as
\ber\label{2.16}
\left(\frac{\tilde{a}'}{r}\right)'=\beta^{2}f^{2}\frac{(\tilde{a}-1)}{r}.
\eer
Since $f>0$, $0<\tilde{a}<1$, we see that $\tilde{a}'/r$ is decreasing for $r>0$. In particular $\tilde{a}'/r$ approaches a finite limit as $r\rightarrow\infty$. We claim that
\ber\label{2.15}
\lim\limits_{r\rightarrow\infty}\frac{\tilde{a}'(B_{0},r)}{r}=0.
\eer
Otherwise, there is an $R_{0}>0$ such that
\ber
\frac{\tilde{a}'(B_{0},r)}{r}\geq c_{0},~~~R_{0}<r<\infty,
\eer
for some constant $c_{0}>0$, which leads to
\ber
\tilde{a}(B_{0},r_{2})-\tilde{a}(B_{0},r_{1})=\int_{r_{1}}^{r_{2}}\tilde{a}'(B_{0},r)\mathrm{d}r\geq\frac{c_{0}}{2}(r_{2}^{2}-r_{1}^{2}),\notag
\eer
for any $r_{2}>r_{1}>R_{0}$. This result indicates that the limit of $\tilde{a}(B_{0},r)$ as $r\rightarrow\infty$ cannot exist, a contradiction. So (\ref{2.15}) is valid.

We now prove that $\tilde{a}(B_0,r)\geq1-\varepsilon$ for any $\varepsilon>0$ as $r\rightarrow\infty$. Otherwise, using (\ref{2.15}) and $f(\infty)=1$, there exists a positive constant $c_1$ such that
\ber
\tilde{a}''(b_0,r)=\frac{\tilde{a}'(b_0,r)}{r}+\beta^{2}f^{2}(1-\tilde{a}(b_0,r))\geq c_1,~r\rightarrow\infty,
\eer
which leads to $a'(B_0,r_{3})\rightarrow\infty$ at infinity. This contradicts the fact $B_{0}\in\mathcal{B}^{0}$. Thus, we have $1-\varepsilon\leq\tilde{a}(B_0,r)\leq1$ for any $\varepsilon>0$ as $r\rightarrow\infty$, namely, $\lim\limits_{r\rightarrow\infty}\tilde{a}(B_0,r)=1$.
\end{proof}

\begin{lemma}\label{2.0}
Let $f(r)$ be given in Proposition \ref{l1.1}. The sets $\mathcal{B}^{-}$ and $\mathcal{B}^{+}$ are both open and nonempty.
\end{lemma}
\begin{proof}
If $B=0$, then from (\ref{2.2})
\begin{align}
\tilde{a}(0,r)&=\frac{\beta}{2}\int_{0}^{r}s\left(\frac{r^{2}}{s^{2}}-1\right)f^{2}(s)(\tilde{a}(s)-1)\mathrm{d}s,\\
\tilde{a}'(0,r)&=\frac{\beta}{2}\int_{0}^{r}\frac{2r}{s}f^{2}(s)(\tilde{a}(s)-1)\mathrm{d}s.
\end{align}
Obviously, $\tilde{a}(0,r)<0$, $\tilde{a}'(0,r)<0$ for all $r>0$ by iteration. Since $\tilde{a}$ and $\tilde{a}'$ continuously depend on the parameter $B$, so when $B_{1}$ is close to 0, we also have
\ber
\tilde{a}(B_{1},r_{1})<0, \tilde{a}'(B_{1},r_{1})<0,
\eer
for some $r_{1}>0$. However, from (\ref{2.3}), we obtain $\tilde{a}(B,r)>0$, $\tilde{a}'(B,r)>0$ initially for any $B>0$. Then there is a point $r_{0}\in(0,r_{1})$ such that $r_{0}$ is a local maximum of $\tilde{a}(B_{1},r)$ and $\tilde{a}'(B_{1},r_{0})=0$. If $\tilde{a}(B_{1},r_{0})=1$, applying the uniqueness theorem for the initial value problem of an ordinary differential equation we have $\tilde{a}(B_{1},r)=1$ for all $r>0$, which is false. If $\tilde{a}(B_{1},r_{0})>1$, such a situation violates the equation (\ref{2.1}).
Thus we are left with the only possible situation, $\tilde{a}(B_{1},r_{0})<1$. Of course, we can find $r_{*}\in(r_{0},r_{1})$ such that $\tilde{a}'(B_{1},r_{*})<0$ and $0<\tilde{a}(B_{1},r)<1$ for any $r\in(0,r_{*}]$, which implies $B_{1}\in\mathcal{B}^{-}$.

It remains to show $\mathcal{B}^{+}\neq\emptyset$. Making the change of scale $r=B^{-\frac{1}{2}}t$ in integral equation (\ref{2.1}) and writing $\hat{a}(t)=\tilde{a}(B^{-\frac{1}{2}}t)$, $\hat{f}(t)=f(B^{-\frac{1}{2}}t)$, we have
\ber\label{2.10}
\hat{a}''(t)-\frac{1}{t}\hat{a}'(t)=\frac{\beta^{2}}{B}\hat{f}^{2}(t)(\hat{a}(t)-1), t\in[0,R_B B^{\frac{1}{2}}).
\eer
Under the assumption of $f(r)$, we obtain
\begin{equation}
\left|\frac{\beta^{2}}{B}\hat{f}^{2}(t)(\hat{a}(t)-1)\right|\leq\frac{\beta^{2}}{B}|\hat{a}(t)-1|,~~~t>0.\notag
\end{equation}
In view of (\ref{00}), there holds $\delta B^{\frac{1}{2}}\rightarrow1$ as $B\rightarrow\infty$. Takeing $B\rightarrow\infty$, we see that
\ber\label{2.11}
\hat{a}''(t)-\frac{1}{t}\hat{a}'(t)=0, ~t\in[0,1],\label{2.11}\\
\hat{a}(t)\sim t^{2}, ~t\rightarrow0.\label{2.12}
\eer
Since the solution of (\ref{2.11})-(\ref{2.12}) is $\hat{a}(t)=t^{2}$, it follows that for small value of $\frac{\beta^{2}}{B}$, $\hat{a}(B,t)=1$ at $t_{0}=1$ and $\hat{a}'(B,t_{0})>0$. So $\mathcal{B}^{+}$ is non-empty. Using the continuous dependence of $\tilde{a}$ and $\tilde{a}'$ on the parameter $B$, $\mathcal{B}^{-}$ and $\mathcal{B}^{+}$ are open set.
\end{proof}

Since the connected set $B>0$ cannot consist of two open disjoint non-empty sets, there must be some value of $B$ in $\mathcal{B}^{0}$ such that $0<\tilde{a}(B,r)<1$, $\tilde{a}'(B,r)\geq0$ for all $r>0$. If $\tilde{a}(B,r)$ is not strictly increasing, it must be a  constant in an interval. So we arrive at a contradiction using (\ref{2.1}) because we would have $\tilde{a}(B,r)=1$ for all $r>0$ which is impossible. Thus, for any $B\in\mathcal{B}^{0}$,
\ber
0<\tilde{a}(B,r)<1, \tilde{a}'(B,r)>0,~r>0.
\eer

\begin{lemma}\label{ll1}
Let $f(r)$ be given in Proposition \ref{l1.1}. If $B_{0}\in\mathcal{B}^{0}$, then for any $0\leq r\leq1$
\ber
\tilde{a}(B_{0},r)\leq M_{2}r^{2},
\eer
where $M_{2}>0$ is a constant.
\end{lemma}
\begin{proof}
From the argument of Lemma \ref{2.0}, there exists $N>0$ large enough such that $\tilde{a}(B,r)$ crosses $1$ when
\ber
\frac{\beta^{2}}{B}<\frac{1}{N}.\nn
\eer
It follows that if $B>\beta^{2}N$, then $B\in\mathcal{B}^{+}$. Since $B_{0}\notin\mathcal{B}^{+}$, it must satisfy $B_{0}\leq\beta^{2}N$.

Since $f(r)\leq M_{1}r$ for $0\leq r\leq1$, there holds
\begin{align}
\left|\tilde{a}(B_{0},r)\right|&\leq B_{0}r^{2}+\frac{\beta^{2}}{2}\int_{0}^{r} \nn s\left(\frac{r^{2}}{s^{2}}-1\right)f^{2}(s)(\tilde{a}(s)+1)\mathrm{d}s\\\nn
&\leq B_{0}r^{2}+\beta^{2}\int_{0}^{r}M_{1}^{2}s^{3}\mathrm{d}s\\\nn
&\leq B_{0}r^{2}+\frac{\beta^{2}M_{1}^{2}}{4}r^{4}\\\nn
&\leq M_{2}r^{2},~~r\leq1,
\end{align}
where we denote $M_{2}=\beta^{2}N+\frac{\beta^{2}M_{1}^{2}}{4}$. The proof of lemma is complete.
\end{proof}

\begin{lemma}
Let $f(r)$ be given in Proposition \ref{l1.1}. The equation (\ref{1.1}) admits a unique solution satisfying $a(0)=0$ and $a(\infty)=1$.
\end{lemma}
\begin{proof}
Suppose that there are two solutions $a_{1}(r)$ and $a_{2}(r)$. Set $\phi(r)=a_{2}(r)-a_{1}(r)$. Obviously, $\phi(0)=0$, $\phi(\infty)=0$ and $\phi(r)$ satisfies the equation
\ber\label{2.18}
\phi''-\frac{1}{r}\phi'-\beta^{2}f^{2}\phi=0.
\eer
We claim that $\phi(r)\equiv0$. If there is a point $r_{0}$ such that $\phi(r_{0})\neq0$, without loss of generality we may assume $\phi(r_{0})>0$. It follows that there exists some point $r_{1}$ such that $\phi(r_{1})>0$, $\phi'(r_{1})=0$ and $\phi''(r_{1})\leq0$. From (\ref{2.18}), we have
\ber
 \phi''(r_{1})=\frac{1}{r_{1}}\phi'(r_{1})+\beta^{2}f^{2}(r_{1})\phi(r_{1})>0,\notag
\eer
which reaches a contradiction. Therefore, we obtain $a_{1}(r)=a_{2}(r)$ for all $r>0$.
\end{proof}

\begin{lemma}\label{l6.2}
Let $f(r)$ be given in Proposition \ref{l1.1}. There holds the asymptotic decay estimates
\begin{equation}\label{2.19}
a(r)=O(r^{\frac{1}{2}}\exp{(-\beta(1-\varepsilon)r)}),~~r\rightarrow\infty.
\end{equation}
\end{lemma}
\begin{proof}
1t is convenient to introduce the shifted field $v(r)$ defined by
\ber
v(r)=r^{-\frac{1}{2}}a(r).
\eer
Clearly, $v(r)>0$ for any $r>0$ and $v(\infty)=0$. Inserting $a(r)=r^{\frac{1}{2}}v(r)$ in (\ref{1.1}), then
\begin{equation}
v''(r)=(\frac{3}{4}\frac{1}{r^{2}}+\beta^{2}f^{2}(r))v(r).\notag
\end{equation}

Now define the comparison function
\begin{equation}\label{2.20}
\eta(r)=c\exp{(-\beta(1-\varepsilon)r)}.
\end{equation}
where $c>0$ is a constant to be chosen later. A simple calculation shows that
\begin{equation}
(v-\eta)''=\beta^{2}(1-\varepsilon)^{2}(v-\eta)
+\left(\frac{3}{4}\frac{1}{r^{2}}+\beta^{2}(f^{2}-(1-\varepsilon)^{2})\right)v.\notag
\end{equation}
Since $f(r)\rightarrow1$ as $r\rightarrow\infty$, there exists a positive constant $r_{\varepsilon}$ sufficiently large such that
\begin{equation}
\frac{3}{4}\frac{1}{r^{2}}+\beta^{2}(f^{2}-(1-\varepsilon)^{2})>0,~~r>r_{\varepsilon}.\notag
\end{equation}
So we arrive at
\begin{equation}\label{2.21}
(v-\eta)''\geq\beta^{2}(1-\varepsilon)^{2}(v-\eta), r>r_{\varepsilon}.
\end{equation}
Choose $c>0$ in the definition of $\eta$ in (\ref{2.20}) to make $(v-\eta)(r_{\varepsilon})\leq0$. Using boundary condition $(v-\eta)(r)\rightarrow0$ as $r\rightarrow\infty$, we obtain by applying the maximum principle in (\ref{2.21}) the result $v(r)<\eta(r)$ for all $r>r_{\varepsilon}$.
\end{proof}

\section{Existence and uniqueness of $g(r)$}
In this section, for the fixed $f(r)$ we will prove the following proposition through several lemmas by shooting the initial value of $g(r)$.

\begin{proposition}\label{l1.2}
Given function $f(r)$ as in Proposition \ref{l1.1}, we can find a unique function $g(r)$ satisfying (\ref{1.2}) and mixed boundary condition $g'(0)=0$, $g(\infty)=0$. In particularly, if $0<\alpha<\sqrt{\frac{1}{2}}$, then $g(r)$ strictly decreasing and  $0<g(r)<\sqrt{1-2\alpha^{2}}$ for all $r>0$. However, if $\alpha\geq\sqrt{\frac{1}{2}}$, then $g(r)\equiv0$.
\end{proposition}

Consider the initial value problem
\ber
g''+\frac{1}{r}g'-\alpha^{2}g-\frac{1}{2}(f^{2}+g^{2}-1)g=0,\label{3.1}\\
g(0)=\lambda,g'(0)=0,\label{3.1'}
\eer
where $\lambda>0$ is an arbitrary constant. We may convert (\ref{3.1}) into the integral equation
\begin{equation}\label{3.2}
g(r)=\lambda+\int_{0}^{r}s(\ln{r}-\ln{s})g(s)\left(\alpha^{2}+\frac{1}{2}(f^{2}(s)+g^{2}(s)-1)\right)\mathrm{d}s.
\end{equation}

For each $\lambda>0$ and $f(r)$ given in Proposition \ref{l1.1}, the boundary value problem (\ref{3.1})-(\ref{3.1'}) admits a unique solution near $r=0$ by Picard iteration. The solution $g(r)$ can be extended uniquely to a maximal interval $[0,R_{\lambda})$ such that either $R_{\lambda}=\infty$ or $|g(r)|\rightarrow\infty$ as $r\rightarrow R_{\lambda}$. In addition, the solution is smooth in $r$ definition domain and has the expansion
\ber\label{3.3}
g(r)=\lambda+O(r^{2})~(r\rightarrow0).
\eer

\begin{lemma}
Let $f(r)$ be given in Proposition \ref{l1.1}. If $\alpha\geq\frac{1}{2}$, then $g(r)=0$ for all $r>0$.
\end{lemma}
\begin{proof}
In view of $\alpha\geq\sqrt{\frac{1}{2}}$, it is clear that for all $r>0$,
\ber\label{3.6}
\alpha^{2}+\frac{1}{2}(f^{2}+g^{2}-1)\geq\frac{1}{2}(f^{2}+g^{2})>0.
\eer
By the maximum principle, we have $g(r)=0$ for all $r\in(0,\infty)$.
\end{proof}

From now on we assume $\alpha<\frac{1}{2}$ and use $g(\lambda,r)$ to denote the solution to the initial value problem (\ref{3.1})-(\ref{3.1'}). Note that for $\lambda\geq\sqrt{1-2\alpha^{2}}$, by iteration we have
\ber\label{3.14}
g(r)\geq\sqrt{1-2\alpha^{2}},~g'(r)\geq0,~r>0,
\eer
which implies that the boundary condition $g(\infty)=0$ is not fulfilled. Then we consider $\lambda\in I=(0,\sqrt{1-2\alpha^{2}})$ and denote the solution as $g(\lambda,r)$. Define parameter sets as
\ber
\begin{aligned}
I_{+}=&\big\{\lambda\in I~|~\text{There exists some}~r_{0}\in(0,R_{\lambda})~\text{such that}
~g'(\lambda,r_{0})=0~\text{and}~g(\lambda,r)>0~\text{for any}~r\in(0,r_{0}]\big\},\notag\\
I_{0}=&\big\{\lambda\in I~|~g'(\lambda,r)<0\mbox{ and } g(\lambda,r)>0~\text{for all}~r>0\big\},\notag\\
I_{-}=&\big\{\lambda\in I~|~\text{There exists some}~r_{0}\in(0,R_{\lambda})~\text{such that}
~g(\lambda,r_{0})=0~\text{and}~g'(\lambda,r)<0~\text{for any}~r\in(0,r_{0}]\big\}.\notag
\end{aligned}
\eer
Clearly,
\ber
I_{+}\cap I_{0}=I_{+}\cap I_{-}=I_{0}\cap I_{-}=\emptyset,~I_{+}\cup I_{0}\cup I_{-}=I.
\eer

\begin{lemma}
Let $\alpha<\sqrt{\frac{1}{2}}$ and $f(r)$ be given in Proposition \ref{l1.1}. If $\lambda_{0}\in I_{0}$, then $g(\lambda_{0},r)$ is a solution to the initial value problem (\ref{3.1})-(\ref{3.1'}) with $g(\lambda_{0},\infty)=0$.
\end{lemma}
\begin{proof}
We only need to prove that $g(\lambda_{0},r)$ enjoys the condition $g(\lambda_{0},\infty)=0$. The fact that the limit
\ber
\lim\limits_{r\rightarrow\infty}g(\lambda_{0},r)=l,
\eer
exists follows immediately from $\lambda_{0}\in I_{0}$. Obviously, $0\leq l<\sqrt{1-2\alpha^{2}}$.

Indeed, if $l>0$, for any $0<\varepsilon<l^{2}$, there exists a positive constant $R_{1}$ such that
\ber
f^{2}(r)>1-\varepsilon,~g^{2}(r)>l^{2},~r>R_{1}.
\eer
Then we see that
\ber
\begin{aligned}
g''(\lambda_0,r)&=-\frac{1}{r}g'(\lambda_0,r)+\left(\alpha^2+\frac{1}{2}(f^{2}+g^{2}-1)\right)g(\lambda_{0},r)\\
&\geq\left(\alpha^2+\frac{1}{2}(f^{2}+g^{2}-1)\right)g(\lambda_0,r)\\
&\geq\frac{l}{2}\left(\alpha^{2}+\frac{1}{2}(l^{2}-\varepsilon)\right),~r>R_{1}.
\end{aligned}
\eer
This implies that $g'(\lambda_{0},r)>0$ for some finite $r>R_1$, which contradicts the fact ${\lambda_{0}}\in I_{0}$. Thus $l=0$ follows.
\end{proof}

\begin{lemma}\label{l3.1}
Let $\alpha<\sqrt{\frac{1}{2}}$ and $f(r)$ be given in Proposition \ref{l1.1}. The set $I_{+}$ and $I_{-}$ are non-empty.
\end{lemma}
\begin{proof}
Let us rewrite the equation (\ref{3.1}) as follows
\ber\label{3.4}
(rg'(r))'=rg(r)\left(\alpha^{2}+\frac{1}{2}(f^{2}(r)+g^{2}(r)-1)\right).
\eer
In view of $\lambda\in I$ and $f(0)=0$, we have
\ber
\alpha^{2}+\frac{1}{2}(f^{2}+g^{2}-1)<0
\eer
for small $r$, which leads to
\ber
(rg')'=rg\left(\alpha^{2}+\frac{1}{2}(f^{2}(r)+g^{2}(r)-1)\right)<0.
\eer
Hence $rg'(\lambda,r)$ is decreasing for small $r$, so $g'(\lambda,r)<0$ initially.

Since $g$ continuously depend on initial value $\lambda$, from (\ref{3.14}), we see that there exists $r_{1}$ such that $g(\lambda_{1},r_{1})\geq\sqrt{1-2\alpha^{2}}$ for $\lambda_{1}\in I$ close enough to $\sqrt{1-2\alpha^{2}}$. From $g'(\lambda,r)<0$ initially, there must be a local minimum point $r_{2}\in(0,r_{1})$ of $g(\lambda_{1},r)$ such that $g'(\lambda_{1},r_{2})=0$ and $g(\lambda_{1},r_{2})>0$, which implies $\lambda_1\in I_{+}$.

In the following, we are to show $I_{-}$ is nonempty. If $\lambda=0$, it is clear that $g(0,r)\equiv0$. So, if $\lambda_{2}$ close enough to 0, then $g(\lambda_{2},r)$ is small for any bounded range of $r$. In such a range, the term in $g^3$ in (\ref{3.4}) becomes negligible, and (\ref{3.4}) becomes effectively
\ber\label{3.7}
(rg')'+\left(\frac{1}{2}(1-f^{2})-\alpha^2\right)gr=0.
\eer
Since $f(0)=0$ and $\alpha<\sqrt{\frac{1}{2}}$, there exist positive constants $\mu$ and $\tilde{r}$ such that
\ber\label{2.8}
0<\mu<\frac{1}{2}(1-f^2)-\alpha,~0<r<\tilde{r}.
\eer
Consider the following initial value problem
\ber
&&h''(r)+\frac{1}{r}h'+\mu h=0,\\
&&h'(0)=0, h(0)=1.
\eer
Assume that $r_0$ is the first zero of the Bessel function $h(r)$. It follows that if $h(0)=\lambda$, then the first zero of $h(r)$ is $\lambda r_0$. Taking $\lambda_3\leq\lambda_2$ and $\lambda_3 r_0\leq\tilde{r}$. Using (\ref{2.8}) and Sturm-Picone comparison theorem,
there exists $r_{3}\leq\lambda_3 r_0$ such that $g(\lambda_{3},r_{3})=0$. Of course, $g'(\lambda_{3},r_{3})<0$ because otherwise for all $r>0$, we have $g(\lambda_{3},r)=0$ which is known to be impossible. Consequently, for any $r\in(0,r_{3}]$, $g'(\lambda_{3},r)<0$, which indicates that $I_{-}$ is nonempty. Continuity can ensure that two sets are open.
\end{proof}

By connectedness, we see that there must be some $\lambda\in I_0$. For this value of $\lambda$, we have that  $0<g(\lambda,r)<\sqrt{1-2\alpha^2}$, $g'(\lambda,r)<0$ for any $r>0$.

\begin{lemma}\label{l4.4}
Let $\alpha<\sqrt{\frac{1}{2}}$ and $f(r)$ be given in Proposition \ref{l1.1}. There holds the asymptotic decay estimate
\ber\label{3.11}
g(r)=O(\exp{(-\alpha(1-\varepsilon)r)}),~~r\rightarrow\infty.
\eer
\end{lemma}
\begin{proof}
To get the decay estimate for $g(r)$, we introduce a comparison function
\begin{equation}\label{3.12}
\sigma(r)=c\exp{(-\alpha(1-\varepsilon)r)},
\end{equation}
where $c>0$ is a constant to be chosen later. A direct calculation shows that
\begin{equation}
\sigma''(r)+\frac{1}{r}\sigma'(r)=\left(\alpha^{2}(1-\varepsilon)^{2}-\frac{\alpha(1-\varepsilon)}{r}\right)\sigma(r),
\end{equation}
so that
\begin{align}
(g-\sigma)''+\frac{1}{r}(g-\sigma)'=&\left(\alpha^{2}(1-(1-\varepsilon)^{2})+\frac{1}{2}(f^{2}+g^{2}-1)\right)g+\frac{\alpha(1-\varepsilon)}{r}\sigma\notag\\
&+\alpha^{2}(1-\varepsilon)^{2}(g-\sigma).\notag
\end{align}
Since $f(r)\rightarrow1$ as $r\rightarrow\infty$, we see that there exists a sufficiently large $r_{\varepsilon}>0$ such that
\begin{equation}
\left(\alpha^{2}(1-(1-\varepsilon)^{2})+\frac{1}{2}(f^{2}+g^{2}-1)\right)g+\frac{\alpha(1-\varepsilon)}{r}\sigma\geq0\notag
\end{equation}
for any $r>r_{\varepsilon}$, which gives us
\begin{equation}\label{3.13}
(g-\sigma)''+\frac{1}{r}(g-\sigma)'\geq\alpha^{2}(1-\varepsilon)^{2}(g-\sigma), ~~r>r_{\varepsilon}.
\end{equation}
Choose the coefficient $c$ in (\ref{3.12}) large enough to make $(g-\sigma)(r_{\varepsilon})\leq0$. Using boundary condition $g(\infty)=0$
 and applying the maximum principle in (\ref{3.13}), we obtain $g(r)<\sigma(r)$ for all $r>r_{\varepsilon}$.
\end{proof}

\begin{lemma}
For $f(r)$ given in Proposition \ref{l1.1}, the solution to equation (\ref{3.1}) with the boundary condition $g'(0)=0$ and $g(\infty)=0$ is unique.
\end{lemma}
\begin{proof}
Suppose to the contrary that there are two different solutions $g_{1}$ and $g_{2}$ satisfying
\ber
g'_{1}(0)=g'_{2}(0)=0,~~~g_{1}(\infty)=g_{2}(\infty)=0.
\eer
Let us rewrite the equation (\ref{3.1}) as follows
\begin{equation}
\frac{(rg')'}{g}=r\left(\alpha+\frac{1}{2}(f^{2}+g^{2}-1)\right).
\end{equation}
Then there holds
\ber\label{3.17}
\frac{(rg_{1}')'}{g_{1}}-\frac{(rg_{2}')'}{g_{2}}=\frac{1}{2}r(g_{1}^{2}-g_{2}^{2}).
\eer
As in \cite{bh}, multiplying (\ref{3.17}) through by $g_{2}^{2}-g_{1}^{2}$ and integrating from zero to infinity, we obtain the identity
\ber
\int_{0}^{\infty}r\left\{\left(g'_{1}-\frac{g_{1}g'_{2}}{g_{2}}\right)^{2}+\left(g'_{2}-\frac{g_{2}g'_{1}}{g_{1}}\right)^{2}\right\}\mathrm{d}r
=-\frac{1}{2}\int_{0}^{\infty}r(g_{1}^{2}-g_{2}^{2})^{2}\mathrm{d}r.
\eer
by using (\ref{3.3}) and (\ref{3.11}). Obviously, we conclude that $g_{1}=g_{2}$ for all $r>0$. Then the lemma follows.
\end{proof}

\section{Existence and uniqueness of $\tilde{f}(r)$}
\setcounter{equation}{0}
In this section we focus on the equation that governs the $f$-component to prove the following proposition.
\begin{proposition}\label{l1.3}
Given functions $f(r)$ as in Proposition \ref{l1.1}, the associated functions $a(r)$ and $g(r)$ obtained in Proposition \ref{l1.1} and \ref{l1.2}, we can find a unique function $\tilde{f}(r)$ satisfying (\ref{1.3}) with boundary conditions $\tilde{f}(0)=0$, $\tilde{f}(\infty)=1$. More precisely, $\tilde{f}(r)$ is strictly increasing, $0<\tilde{f}(r)<1$ for all $r>0$, $r^{-1}\tilde{f}\leq M_{3}$ for $r\leq1$, where $M_{3}$ is a positive constant.
\end{proposition}

Let us rewrite (\ref{1.3}) as
\begin{equation}\label{4.1}
\tilde{f}''(r)+\frac{1}{r}\tilde{f}'(r)-\frac{1}{r^{2}}\tilde{f}(r)=\frac{1}{r^{2}}(a^{2}(r)-1)\tilde{f}(r)+\frac{1}{2}(\tilde{f}^{2}(r)+g^{2}(r)-1)\tilde{f}(r).
\end{equation}
From $\tilde{f}(0)=0$, (\ref{4.1}) can be converted into the integral equation
\begin{equation}\label{4.2}
\tilde{f}(r)=Dr+\frac{1}{2}\int_{0}^{r}r\left(1-\frac{s^{2}}{r^{2}}\right)\tilde{f}(s)\left\{\frac{1}{s^{2}}(a^{2}(s)- 1)+\frac{1}{2}(\tilde{f}^{2}(s)+g^{2}(s)-1)\right\}\mathrm{d}s,
\end{equation}
where $D>0$ is an arbitrary constant.


\begin{lemma}
 Let $D>0$ and $a(r)$, $g(r)$ be obtained in Proposition \ref{l1.1} and \ref{l1.2}, respectively. Then the equation (\ref{4.1}) subject to the initial condition $\tilde{f}(0)=0$ can be solved at least for sufficiently small $r$, more precisely,
\ber\label{4.3}
\tilde{f}(r)=Dr+O(r^{3}),~~r\rightarrow0.
\eer
\end{lemma}
\begin{proof}
The integral equation (\ref{4.2}) can be solved by Picard iteration. To this end, we denote
\begin{align}
\tilde{f}_{0}(r)&=Dr,\notag \\
\tilde{f}_{n}(r)&=Dr+\frac{1}{2}\int_{0}^{r}r\left(1-\frac{s^{2}}{r^{2}}\right)\tilde{f}_{n-1}\left\{\frac{1}{s^{2}}(a^{2}-1)+\frac{1}{2}(\tilde{f}_{n-1}^{2}+g^{2}-1)\right\}\mathrm{d}s,~n=1,2,...\notag
\end{align}

Since for any $r\leq1$, $|a(r)-1|\leq M_2r^{2}$, we obtain the following upper bound estimate,
\begin{align}
|\tilde{f}_{1}-\tilde{f}_0|&\leq \frac{1}{2}\int_{0}^{r}r(1-\frac{s^{2}}{r^{2}})Ds\left(2M_2+\frac{1}{2}(D^{2}s^{2}+1)\right)\mathrm{d}s\notag\\
&\leq\frac{r}{2}\int^{r}_{0}(2M_2+1)Ds\mathrm{d}s\notag\\
&\leq (2M_2+1)\frac{D}{4}r^3,~~~r\leq\min\left\{1,\frac{1}{D}\right\}. \notag
\end{align}
It follows that for any $r\leq\min\left\{1,\frac{1}{D},\sqrt{\frac{4}{2M_2+1}}\right\}$,
\ber
|\tilde{f}_1|\leq Dr+(2M_2+1)\frac{D}{4}r^3\leq2Dr.\notag
\eer
Similarly, we have
\begin{align}
|\tilde{f}_{2}-\tilde{f}_0|&\leq \frac{1}{2}\int_{0}^{r}r(1-\frac{s^{2}}{r^{2}})2Ds\left(2M_2+\frac{1}{2}(4D^{2}s^{2}+1)\right)\mathrm{d}s\notag\\
&\leq \frac{r}{2}\int^{r}_{0}2D\left(2M_2+\frac{5}{2}\right)s\mathrm{d}s\notag\\
&\leq \left(2M_2+\frac{5}{2}\right)\frac{D}{2}r^3,~~~r\leq\min\left\{1,\frac{1}{D},\sqrt{\frac{4}{2M_2+1}}\right\}. \notag
\end{align}
Then there holds
\ber
|\tilde{f}_2|\leq Dr+\left(2M_2+\frac{5}{2}\right)\frac{D}{2}r^3\leq2Dr,
\eer
for any $r\leq\min\left\{1,\frac{1}{D},\sqrt{\frac{2}{2M_2+\frac{5}{2}}}\right\}$. As a consequence, we find
\begin{align}
|\tilde{f}_n-\tilde{f}_0|&\leq\left(2M_2+\frac{5}{2}\right)\frac{D}{2}r^3,~n=1,2,...\label{4.4}\\
|\tilde{f}_n|&\leq2Dr,~n=1,2,...\label{4.5}
\end{align}
for any $r\leq\min\left\{1,\frac{1}{D},\sqrt{\frac{2}{2M_2+\frac{5}{2}}}\right\}$.

Using (\ref{4.4}) and (\ref{4.5}), a simple calculation shows
\begin{align}
|\tilde{f}_{n+1}-\tilde{f}_{n}|
&\leq\frac{1}{2}\int_{0}^{r}r(1-\frac{s^{2}}{r^{2}})|f_{n}-f_{n-1}|
\left(\frac{|a^2-1|}{s^2}+\frac{1}{2}(f_{n}^{2}+f_{n}f_{n-1}+f_{n-1}^{2})+\frac{1}{2}(g_0^2-1)\right)\mathrm{d}s\notag\\
&\leq\frac{r}{2} \int^{r}_{0}\left(2M_2+\frac{13}{2}\right)|f_{n}-f_{n-1}|\notag\\
&\leq (2M_2+1)\frac{Dr}{2}\left(\frac{M_{**}}{2}\right)^n\frac{r^{2(n+1)}}{2(n+1)!!}\notag\\
&\leq\frac{(\frac{M_{**}}{2}r^2)^{n+1}}{2(n+1)!!},~~~n=1,2...\notag
\end{align}
where $M_{**}=2M_2+\frac{13}{2}$. Setting
\ber\label{4.6}
\delta_0=\min\left\{1,\frac{1}{D},\sqrt{\frac{2}{M_{**}}}\right\}.
\eer
In summary, we obtain
\ber
\sum\limits_{n=0}^{\infty}(\tilde{f}_{n+1}-\tilde{f}_{n})\leq\sum\limits_{n=1}^{\infty}\frac{(\frac{M_{**}}{2}r^2)^{n+1}}{2(n+1)!!},~r\leq\delta_0.
\eer
It is seen that the series $\tilde{f}_{0}+\sum\limits_{n=1}^{\infty}(\tilde{f}_{n}-\tilde{f}_{n-1})$ converges uniformly on $[0,\delta_0]$. Hence we may assume
$\tilde{f}_{n}\rightarrow\tilde{f}$ as $n\rightarrow\infty$ in $[0,\delta_0]$. Furthermore, The solution $\tilde{f}(r)$ can be extended uniquely to a maximal interval $[0, R_{D})$. Taking $n\rightarrow\infty$ on (\ref{4.4}), we obtain (\ref{4.3}) as expected.
\end{proof}

We now denote the solution as $\tilde{f}(D,r)$ and define parameter sets as follows.
\ber
\begin{aligned}
\mathcal{D}^{-}=&\big\{D>0~|~\text{There exists a}~r_{*}\in(0,R_{D})~\text{such that}
~\tilde{f}'(D,r_{*})<0~\text{and for any}~r\in(0,r_{*}],0<\tilde{f}(D,r)<1\big\},\notag\\
\mathcal{D}^{0}=&\big\{D>0~|~\tilde{f}'(D,r)\geq0\mbox{ and } 0<\tilde{f}(D,r)<1~\text{for all}~r>0\big\},\notag\\
\mathcal{D}^{+}=&\big\{D>0~|~\mbox{There exists a}~r^{*}\in(0,R_{D})~\text{such that}~\tilde{f}(D,r_{*})=1~\text{and for any}~r\in(0,r^{*}], \tilde{f}'(D,r)\geq0\big\}.\notag
\end{aligned}
\eer
Clearly,
\ber
\mathcal{D}^{-}\cap\mathcal{D}^{0}=\mathcal{D}^{0}\cap\mathcal{D}^{+}=\mathcal{D}^{-}\cap\mathcal{D}^{+}=\emptyset, \mathcal{D}^{-}\cup\mathcal{D}^{0}\cup\mathcal{D}^{+}=(0,\infty).
\eer

\begin{lemma}
Let $(f(r),a(r),g(r))$ be given in Proposition \ref{l1.3}. If $D_{0}\in\mathcal{D}^{0}$, then $\tilde{f}(D_{0},r)$ is a solution to the equation (\ref{4.1}) with $\tilde{f}(D_{0},0)=0$, $\tilde{f}(D_{0},\infty)=1$.
\end{lemma}
\begin{proof}
For any $D_{0}\in\mathcal{D}^{0}$, it is easily seen that the limit $\lim\limits_{r\rightarrow\infty}\tilde{f}(D_{0},r)=l$ exists and $0<l\leq1$. We now prove $l=1$. Assuming that $l<1$, from $a(\infty)=0$, $g(\infty)=0$, we have
\ber
\tilde{f}''(D_{0},r)\sim\frac{1}{2}(l^{2}-1)\tilde{f}(D_{0},r),
\eer
which gives a contradiction with $D_{0}\in\mathcal{D}^{0}$.  Then the proof of the lemma is complete.
\end{proof}

\begin{lemma}\label{l4.0}
Let $f(r)$ be given in Proposition \ref{l1.1}, $a(r)$, $g(r)$ obtained in Proposition \ref{l1.1} and \ref{l1.2}, respectively. The set $\mathcal{D}^{-}$ and $\mathcal{D}^{+}$ are both open and non-empty.
\end{lemma}
\begin{proof}
If $D=0$, then clearly $\tilde{f}=0$ for all $r>0$ by iteration. Using the continuous dependence of $\tilde{f}$ on $D$, $\tilde{f}(D_{1},r)$ is also small for any bounded range of $r$ when $D_{1}>0$ is close to 0. In such a range, the term $\tilde{f}^{3}$ in (\ref{4.1}) can be negligible, so that (\ref{4.1}) becomes effectively
\ber\label{4.7}
\tilde{f}''+\frac{1}{r}\tilde{f}'+\left(\frac{1}{2}(1-g^{2})-\frac{a^{2}}{r^{2}}\right)\tilde{f}=0.
\eer
Let us consider the Bessel equation
\ber\label{4.8}
h_{1}''+\frac{1}{r}h'_{1}+\left(\alpha^{2}-\frac{1}{r^{2}}\right)h_{1}=0,
\eer
with $h_{1}\sim r$ as $r\rightarrow0$, and rewrite the equation (\ref{4.7}) and (\ref{4.8}) as
\begin{align}
(rh'_{1})'+\left(\alpha^{2}-\frac{1}{r^{2}}\right)rh_{1}&=0,\label{4.9}\\
(r\tilde{f}')'+\left(\frac{1}{2}(1-g^{2})-\frac{a^{2}}{r^{2}}\right)r\tilde{f}&=0.\label{4.10}
\end{align}
Since $0<a(r)<1$, $0\leq g(r)<\sqrt{1-2\alpha^{2}}$ for any $r>0$, we have
\ber
\alpha^{2}-\frac{1}{r^{2}}-\frac{1}{2}(1-g^{2})+\frac{a^{2}}{r^{2}}<0.
\eer
Applying the Sturm-Picone comparison theorem to (\ref{4.9}) and (\ref{4.10}),  we obtain that $\tilde{f}(r)$ have more zero points than $h(r)$.
Then there exists $r_{*}>0$ such that $\tilde{f}'(D_{1},r_{*})<0$ and for any $r\in(0,r_{*})$, $0<\tilde{f}(D_{1},r)<1$. Therefore $\mathcal{D}^{-}$ is non-empty.

In order to prove $\mathcal{D}^{+}\neq\emptyset$, we introduce a transformation to consider a modified variable $t=Dr$. Writing $\hat{f}(t)=\tilde{f}(D^{-1}t)$, $\hat{a}(t)=a(D^{-1}t)$ and $\hat{g}(t)=g(D^{-1}t)$, the new variables recast (\ref{4.1}) into the form
\ber\label{4.12}
\hat{f}''+\frac{1}{t}\hat{f}'-\frac{1}{t^{2}}\hat{f}=\frac{1}{t^{2}}(\hat{a}^{2}-1)+\frac{1}{2D^{2}}(\hat{f}^{2}+\hat{g}^{2}-1),~t\in[0,R_D D)
\eer
Since $1-a(r)\leq M_{2}r^{2}$, $0\leq r\leq1$, we arrive at
\ber
1-\hat{a}(t)\leq\frac{M_{2}t^{2}}{D^{2}},~0\leq t\leq D.
\eer
Then
\ber
\left|\frac{\hat{a}^{2}-1}{t^{2}}\right|\leq\frac{2(1-\hat{a})}{t^{2}}\leq\frac{2M_{2}}{D^{2}},~0\leq t\leq D.
\eer
From (\ref{4.6}), we have $\delta_0 D\rightarrow1$ as $D\rightarrow\infty$. Thus, letting $D\rightarrow\infty$, we get
\begin{align}
\hat{f}''(t)+\frac{1}{t}\hat{f}'(t)-\frac{1}{t^{2}}\hat{f}(t)=0,~t\in[0,1],\\
\hat{f}(t)\sim t,t\rightarrow0.
\end{align}
The solution of the above equation is $\hat{f}(t)=t$. Thus, for small value of $\frac{2M_{2}}{D^{2}}$, $\hat{f}(t)=1$ at $t_{0}=1$ and $\hat{f}'(t_{0})>0$, which implies that $\mathcal{D}^{+}$ is non-empty. Using the continuous dependence of $\tilde{f}$ and $\tilde{f}'$ on the parameter $D$, $\mathcal{D}^{-}$ and $\mathcal{D}^{+}$ are open set.
\end{proof}

By connectedness, there must be some value of $D\in\mathcal{D}^{0}$. Indeed, $\tilde{f}$ strictly increases in $(0,\infty)$. If $\tilde{f}$ is nondecreasing, the existence of $0<r_1<r_2<\infty$ would make $\tilde{f}$ a constant over the interval $[r_1, r_2]$. This contradicts the uniqueness theorem for ODEs. Consequently, for any $D\in\mathcal{S}^{D}_{2}$,
\ber
0<\tilde{f}(D,r)<1,~\tilde{f}'(D,r)>0,~r>0.
\eer

\begin{lemma}
Let $f(r)$ be given in Proposition \ref{l1.1}, $a(r)$, $g(r)$ obtained in Proposition \ref{l1.1} and \ref{l1.2}, respectively. If $D_{0}\in\mathcal{D}^{0}$, then for any $r\leq1$,
\ber
\tilde{f}(D_{0},r)\leq M_{3}r,
\eer
where $M_{3}$ is a positive constant.
\end{lemma}
\begin{proof}
From the proof of Lemma \ref{l4.0}, there exists some constant, still denoted by $N$, such that if $\frac{2M_{2}}{D^{2}}<\frac{1}{N}$, then $D\in\mathcal{D}^{+}$. Since $D_0\notin\mathcal{D}^{+}$, it must satisfy
\ber\label{4.13}
D_{0}\leq(2M_{2}N)^{\frac{1}{2}}.
\eer
Using $|a(r)-1|\leq M_{2}r$ for any $r\leq1$, we have
\begin{align}
\tilde{f}(D_{0},r)
&\leq D_{0}r+\frac{r}{2}\int_{0}^{r}(2M_{2}+1)\mathrm{d}s\notag\\
&\leq D_{0}r+(2M_{2}+1)\frac{r^2}{2}\notag\\
&\leq M_{3}r,~r\leq1,\notag
\end{align}
where $M_{3}=(2M_{2}N)^{\frac{1}{2}}+(M_{2}+\frac{1}{2})$. Then the lemma follows.
\end{proof}

\begin{lemma}\label{l5.5}
Let $f(r)$ be given in Proposition \ref{l1.1}, $a(r)$, $g(r)$ obtained in Proposition \ref{l1.1} and \ref{l1.2}, respectively. There holds the asymptotic estimate
\ber\label{4.16}
\tilde{f}(r)=1+O(\exp{(-\gamma(1-\varepsilon)r)}),~~r\rightarrow\infty,
\eer
where $\gamma=\min\{1,2\alpha,2\beta\}$.
\end{lemma}
\begin{proof}
Set $\psi(r)=1-\tilde{f}(r)$. Obviously, $\psi(r)>0$ for all $r>0$ and $\psi(r)$ satisfies the following equation
\begin{equation}
\psi''+\frac{1}{r}\psi'=\left(\frac{a^{2}}{r^{2}}+\frac{1}{2}(2-\psi)(1-\psi)
+\frac{1}{2}g^{2}\right)\psi-\frac{a^{2}}{r^{2}}-\frac{g^{2}}{2}.\nn\\
\end{equation}
Now we introduce a comparison function
\begin{equation}\label{4.17}
\eta(r)=c\exp{(-\gamma(1-\varepsilon)r)},
\end{equation}
where $c>0$ and $\gamma=\min\{1,2\alpha,2\beta\}$.
Then,
\begin{align}
(\psi-\eta)''+\frac{1}{r}(\psi-\eta)'=&\left(\frac{a^{2}}{r^{2}}+\frac{1}{2}(2-\psi)(1-\psi)+\frac{1}{2}g^{2}-\gamma^{2}(1-\varepsilon)^{2}\right)\psi\notag\\
&+\gamma^{2}(1-\varepsilon)^2(\psi-\eta)+\frac{\gamma(1-\varepsilon)}{r}\eta-\frac{a^{2}}{r^{2}}-\frac{1}{2}g^{2}.\notag
\end{align}
Using the estimate for $a(r)$, $g(r)$ stated in (\ref{2.19}) and (\ref{3.11}), we see that there is a sufficiently large $r_{\varepsilon}>0$ such that
\ber
\left(\frac{a^{2}}{r^{2}}+\frac{1}{2}(2-\psi)(1-\psi)+\frac{1}{2}g^{2}-\gamma^{2}(1-\varepsilon)^{2}\right)\psi
+\frac{\gamma(1-\varepsilon)}{r}\eta-\frac{a^{2}}{r^{2}}-\frac{1}{2}g^{2}\geq0\notag
\eer
for any $r>r_{\varepsilon}$. So we get
\begin{equation}\label{4.18}
(\psi-\eta)''+\frac{1}{r}(\psi-\eta)'\geq\gamma^{2}(1-\varepsilon)^2(\psi-\eta), r>r_{\varepsilon}.
\end{equation}
Choose the coefficient $c$ in (\ref{4.17}) large enough to make $(\psi-\eta)(r_{\varepsilon})\leq0$. Since $(\psi-\eta)\rightarrow0$ as $r\rightarrow\infty$, applying the maximum principle in (\ref{4.18}), we obtain that $\psi(r)<\eta(r)$ for all $r>r_{\varepsilon}$.
\end{proof}

\begin{lemma}
Let $f(r)$ be given in Proposition \ref{l1.1}, $a(r)$, $g(r)$ obtained in Proposition \ref{l1.1} and \ref{l1.2}, respectively. The solution to problem (\ref{4.1}) with the boundary condition $\tilde{f}(0)=0$ and $\tilde{f}(\infty)=0$ is unique.
\end{lemma}
\begin{proof}
Suppose that there are two different solutions $\tilde{f}_{1}$ and $\tilde{f}_{2}$ satisfying
\ber
\tilde{f}_{1}(0)=\tilde{f}_{2}(0)=0,~~~\tilde{f}_{1}(\infty)=\tilde{f}_{2}(\infty)=1.\nn
\eer
Let us rewrite the equation (\ref{4.1}) as follows
\begin{equation}
\frac{(r\tilde{f}')'}{\tilde{f}}=\frac{(1-a)^{2}}{r}+\frac{1}{2}(\tilde{f}^{2}+g^{2}-1)r.\nn
\end{equation}
Then we have
\ber\label{3.9}
\frac{(r\tilde{f}_{1}')'}{\tilde{f}_{1}}-\frac{(r\tilde{f}_{2}')'}{\tilde{f}_{2}}=\frac{1}{2}r(\tilde{f}_{1}^{2}-\tilde{f}_{2}^{2}).
\eer
Multiplying (\ref{3.9}) by $\tilde{f}_{2}^{2}-\tilde{f}_{1}^{2}$ and integrating over $(0,\infty)$, we obtain the identity
\begin{equation}\label{3.10}
\int_{0}^{\infty}r\left\{\left(\tilde{f}'_{1}-\frac{\tilde{f}_{1}\tilde{f}'_{2}}{\tilde{f}_{2}}\right)^{2}+\left(\tilde{f}'_{2}-\frac{\tilde{f}_{2}\tilde{f}'_{1}}{\tilde{f}_{1}}\right)^{2}\right\}\mathrm{d}r\\
=-\frac{1}{2}\int_{0}^{\infty}r(\tilde{f}_{1}^{2}-\tilde{f}_{2}^{2})^{2}\mathrm{d}r.
\end{equation}
where we have used (\ref{4.3}) and (\ref{4.16}). It is clear that $\tilde{f}_{1}=\tilde{f}_{2}$ for all $r>0$.
\end{proof}

\section{The proof of Theorem 1.1}
In this section, we complete the proof of Theorem \ref{t1} by the Schauder fixed point theorem.

Firstly, we define the space $\mathrm{X}$ as follows
\ber
\mathrm{X}=\big\{f(r)\in C[0,\infty)|~r^{-k}(1+r^{k})f(r)~\text{bounded},0<k<1\big\},\notag
\eer
equipped with the norm
\ber
\|f\|_{\mathrm{X}}=\sup_{r\in[0,\infty)}\{|r^{-k}(1+r^{k})f(r)|\}.
\eer
It is easy to check that $\mathrm{X}$ is indeed a Banach space.
We denote the subset $\mathcal{S}$ of $\mathrm{X}$ by
\ber
\mathcal{S}=\big\{f(r)\in\mathrm{X}|~|r^{-1}f|\leq M,~\text{for}~r\leq1;~f(r)~\text{is increasing};~f(\infty)=1;~0\leq f(r)\leq1\big\},\notag
\eer
where $M=\max\{M_{1},M_{3}\}$. Here, $f(r)$ is chosen in accordance with Proposition \ref{l1.3}. It is straightforward to see that the set $\mathcal{S}$ is
non-empty, bounded, closed, and convex. The previous discussion allows us to define the mapping $T:f\mapsto\tilde{f}$ on $\mathrm{X}$.
\begin{lemma}\label{l5.1}
For the mapping $T:f\mapsto\tilde{f}$,
\item{(i)} The mapping $T$ takes $\mathcal{S}$ into itself;
\item{(ii)} $T$ is compact;
\item{(iii)} $T$ is continuous.
\end{lemma}
\begin{proof}
Part $(i)$ has already established in Proposition \ref{l1.1}, \ref{l1.2} and \ref{l1.3}.

We now concentrate on the proof of Part $(ii)$. For any $f\in S$, then $\tilde{f}$ satisfies (\ref{4.2}). It is straightforward to show that
\begin{equation}
\tilde{f}'_{n}(r)=D_{0}+\frac{1}{2}\int_{0}^{r}\left(1+\frac{s^{2}}{r^{2}}\right)\tilde{f}_{n}(s)\left\{\frac{1}{s^{2}}(a^{2}-1)+\frac{1}{2}(\tilde{f}_{n}^{2}+g^{2}-1)\right\}\mathrm{d}s\notag
\end{equation}
is bounded in any compact sub-interval of $(0,\infty)$. Using the mean value theorem, the sequence $\{\tilde{f}_{n}\}$ is equicontinuous in any compact sub-interval of $(0,\infty)$. Thus, by the Ascoli-Arzel\`{a} theorem, we may assume without loss of generality that $\{\tilde{f}_{n}\}$ converges uniformly to an element $\tilde{f}\in\mathcal{S}$ over any compact sub-interval of $(0,\infty)$ as $n\rightarrow\infty$. Consequently, for any $0<\delta<R<\infty$, $\tilde{f}_{n}\rightarrow\tilde{f}$ uniformly over $[\delta,R]$ as $n\rightarrow\infty$.

We only need to prove that when $r\rightarrow0$ and $r\rightarrow\infty$,
\ber
\|\tilde{f}_{n}-\tilde{f}\|_{X}\rightarrow0.
\eer
Since $|\tilde{f}_{n}|\leq Mr$, $|\tilde{f}|\leq Mr$ for any $r\leq1$, taking $0<\delta<1$, we have
\begin{equation}\label{4.14}
\sup_{r\in[0,\delta)}|r^{-k}(1+r^{k})(\tilde{f}_{n}-\tilde{f})|\leq2\sup_{r\in[0,\delta)}|r^{-k}(\tilde{f}_{n}-\tilde{f})|\leq4M\delta^{1-k},
\end{equation}
which approaches zero as $\delta\rightarrow0$.

On the other hang, since $\tilde{f}_{n}$ decaying exponentially fast at infinity, we see that for any $\varepsilon>0$, there exists $R>0$ large enough independent of $n$ such that
\ber
\sup_{r\in(R,\infty)}|r^{-k}(1+r^{k})(\tilde{f}_{n}-\tilde{f})|<\varepsilon.
\eer
Therefore $\{\tilde{f}_{n}\}$ is uniformly convergent to $\tilde{f}$, and the mapping $T$ is compact.

We now show that $T$ is continuous. It is sufficient to show that for any $f_{1}$, $f_{2}\in S$, if $\|f_{1}-f_{2}\|_{\mathrm{X}}\rightarrow0$, there holds
\ber
\|T(f_{1})-T(f_{2})\|_{\mathrm{X}}\rightarrow0.
\eer
Since $T(f_{1})(\infty)=T(f_{2})(\infty)$, then for any $\varepsilon>0$ there is a positive constant $R_{1}$ large enough such that
\ber
\sup_{r\in[R_{1},\infty)}|r^{-k}(1+r^{k})(T(f_{1})-T(f_{2}))|<\varepsilon.
\eer
For the given $\varepsilon$, in view of (\ref{4.14}), there exists $0<\delta_{1}<1$ such that
\ber
\sup_{r\in(0,\delta_{1})}|r^{-k}(1+r^{k})(T(f_{1})-T(f_{2}))|<\varepsilon.
\eer
Using the continuity of $T$ at $[\delta_1,R_1]$, we see that $T$ is continuous on $\mathrm{X}$.
\end{proof}

{\bf proof of Theorem \ref{t1}} According to Lemma \ref{l5.1}, the mapping $T$ satisfies the conditions of the
Schauder fixed point theorem. Consequently, we can find a solution pair $(a,g,f)$ to the mixed boundary value problem (\ref{1.1})-(\ref{1.5}).

In addition, from (\ref{1.2}) and (\ref{1.3}), $1-f^{2}-g^{2}$ satisfies
\ber
\begin{aligned}
&-(1-f^{2}-g^{2})''-\frac{1}{r}(1-f^{2}-g^{2})'+(g^{2}+f^{2})(1-f^{2}-g^{2})\nn\\
=&2\alpha^{2}g^{2}+\frac{2}{r^{2}}(1-a)^{2}f^{2}+2g'^{2}+2f'^{2},\notag\\
\end{aligned}
\eer
with boundary condition $(1-f^{2}-g^{2})(0)>0$, $(1-f^{2}-g^{2})(\infty)=0$.
Using the maximum principle we arrive at $f^{2}+g^{2}<1$, for all $r>0$. Moreover, the asymptotic estimates at infinity have been established in lemmas \ref{l6.2}, \ref{l4.4} and \ref{l5.5}. Then the Theorem \ref{t1} follows.

\end{document}